\documentclass[a4paper,12pt,reqno]{amsart} 

\usepackage{amssymb,latexsym}
\usepackage{amsmath,amsthm}
\usepackage{enumerate}
\usepackage[mathscr]{eucal}
\usepackage{subcaption}
\usepackage{wrapfig}
\usepackage{graphicx}
\usepackage{multicol}
\usepackage{url}

\usepackage{graphics} 

\theoremstyle{plain}
\newtheorem{theorem}{\indent Theorem}[section]

\newtheorem{lemma}[theorem]{\indent Lemma}
\theoremstyle{definition}
\newtheorem{definition}[theorem]{\indent Definition}

\theoremstyle{remark}
\newtheorem{remark}[theorem]{\indent Remark}

\allowdisplaybreaks

\newcommand{\wlogic}{\textup{$\mathbf{iE}^{\mathsf{W}}$\thinspace}}
\newcommand{\wwlogic}{\textup{$\mathbf{iEC}^{\mathsf{W}}$\thinspace}}
\newcommand{\wwwlogic}{\textup{$\mathbf{iM}^{\mathsf{W}}$\thinspace}}
\newcommand{\wwwwlogic}{\textup{$\mathbf{iK}^{\mathsf{W}}$\thinspace}}

\newcommand{\alog}{\textup{$\Box$\textbf{Log}\thinspace}}
\newcommand{\blog}{\textup{$\Diamond$\textbf{Log}\thinspace}}
\newcommand{\clog}{\textup{$\blacksquare$\textbf{Log}\thinspace}}
\newcommand{\dlog}{\textup{$\bullet$\textbf{Log}\thinspace}}

\newcommand{\pn}{\textup{\textbf{pn}\thinspace}}
\newcommand{\wpn}{\textup{$\mathbf{iE}^{\mathsf{W}}$\textbf{pn}\thinspace}}
\newcommand{\wwpn}{\textup{$\mathbf{iEC}^{\mathsf{W}}$\textbf{pn}\thinspace}}
\newcommand{\wwwpn}{\textup{$\mathbf{iM}^{\mathsf{W}}$\textbf{pn}\thinspace}}
\newcommand{\wwwwpn}{\textup{$\mathbf{iK}^{\mathsf{W}}$\textbf{pn}\thinspace}}

\newcommand{\apn}{\textup{$\Box$\textbf{pn}\thinspace}}
\newcommand{\bpn}{\textup{$\Diamond$\textbf{pn}\thinspace}}
\newcommand{\cpn}{\textup{$\blacksquare$\textbf{pn}\thinspace}}
\newcommand{\dpn}{\textup{$\bullet$\textbf{pn}\thinspace}}

\newcommand{\canpn}{\textup{$\mathbf{iL}^{\mathsf{W}}$\textbf{can-pn}\thinspace}}
\newcommand{\wcpn}{\textup{$\mathbf{iE}^{\mathsf{W}}$\textbf{can-pn}\thinspace}}
\newcommand{\wwcpn}{\textup{$\mathbf{iEC}^{\mathsf{W}}$\textbf{can-pn}\thinspace}}
\newcommand{\wwwcpn}{\textup{$\mathbf{iM}^{\mathsf{W}}$\textbf{can-pn}\thinspace}}
\newcommand{\wwwwcpn}{\textup{$\mathbf{iK}^{\mathsf{W}}$\textbf{can-pn}\thinspace}}

\newcommand{\acanpn}{\textup{$\Box$\textbf{can-pn}\thinspace}}
\newcommand{\bcanpn}{\textup{$\Diamond$\textbf{can-pn}\thinspace}}
\newcommand{\ccanpn}{\textup{$\blacksquare$\textbf{can-pn}\thinspace}}
\newcommand{\dcanpn}{\textup{$\bullet$\textbf{can-pn}\thinspace}}

\newcommand{\ipc}{\textup{\textbf{IPC}\thinspace}}
\newcommand{\ntax}{\textup{$\mathtt{WE}$\thinspace}}
\newcommand{\wcax}{\textup{$\mathtt{WC}$\thinspace}}
\newcommand{\tax}{\textup{$\mathtt{T}$\thinspace}}

\newcommand{\rew}{\textup{$\mathtt{REW}$\thinspace}}
\newcommand{\wnec}{\textup{$\mathtt{RMW}$\thinspace}}
\newcommand{\modpon}{\textup{$\mathtt{MP}$\thinspace}}
\newcommand{\ext}{\textup{$\mathtt{RE}$\thinspace}}

\begin{document}
\null
\vskip 1.8truecm

\title[A note on the intuitionistic logic of false belief]
{A note on the intuitionistic \\ logic of false belief} 
\author{Tomasz Witczak}    
\address{Institute of Mathematics\\ Faculty of Science and Technology
\\ University of Silesia\\ Bankowa~14\\ 40-007 Katowice\\ Poland}
\email{tm.witczak@gmail.com} 


\begin{abstract}In this paper we analyse logic of false belief in intuitionistic setting. This logic, studied in its classical version by Steinsvold, Fan, Gilbert and Venturi, describes the following situation: a formula $\varphi$ is not satisfied in a given world, but we still believe in it (or we think that it should be accepted). Another interpretations are also possible: e.g. that we do not accept $\varphi$ but it is imposed on us by a kind of council or advisory board. From the mathematical point of view, the idea is expressed by an adequate form of modal operator $\mathsf{W}$ which is interpreted in relational frames with neighborhoods. We discuss monotonicity of forcing, soundness, completeness and several other issues. We present also some simple systems in which confirmation of previously accepted formula is modelled. 
\end{abstract}

\subjclass{Primary: 03B45, 03B20 ; Secondary: 03A10}
\keywords{Intuitionistic modal logic, non-normal modal logic, neighborhood semantics}


\maketitle


\section{Introduction}

Logic of false belief was studied e.g. by Steinsvold \cite{steinsvold}, Gilbert and Venturi \cite{gilbert} or Fan \cite{fan}. Those authors obtained several interesting results concerning completeness and expressivity. Their propositional systems were based on classical modal logics (i.e. with the law of the excluded middle). As for the semantics, they used relational (Kripke) and neighborhood frames.

In general, the idea is to describe the following situation: $\varphi$ is false but it is still believed (see \cite{gilbert}). This concept is expressed by the very definition of forcing: if $w$ is a possible world then $w \Vdash \mathsf{W}\varphi$ $\Leftrightarrow$ $w \nVdash \varphi$ and $V(\varphi) \in \mathcal{N}_{w}$. The fact that $\varphi$ is erroneously taken for true, is modeled by the second part of this definition: we see that $V(\varphi)$ is among our neighborhoods. 

Other interpretations are also possible. For example, we deny $\varphi$ (in a given world) but it is imposed on us by a kind of council or advisory board. We are \emph{encouraged} to accept $\varphi$, at least in a particular world, because the set of all worlds accepting $\varphi$ is one of our neighborhoods. This means that $V(\varphi)$ gathers worlds (and thus situations or circumstances) which are similar to our present situation, hence maybe we should rethink our opinion on $\varphi$. Also, we can identify possible worlds with different people, accepting (or not) various formulas. Then $w$-neighborhoods can be considered as (more or less) credible groups of advisors or lustrators. 

We can assume that our worlds are pre-ordered and if $w \leq v$, then $v$ accepts at least everything which was previously approved by $w$. Now our model becomes intuitionistic: we have persistence of truth (with respect to $\leq$). This approach will be studied in the present paper. We are interested mostly in completeness and monotonicity of forcing. We show several intuitionistic versions of classical false belief systems. We discuss restrictions which can or have to be imposed on neighborhoods. We also point out some subtle limitations and advantages of intuitionistic framework in the context of minimal, maximal and intermediate canonical models. 

Finally, we discuss four possible interpretations of intuitionistic modal logic with axiom $T$: again, in terms of confirmation or denial. One of this interpretations leads us to the logic of unknown truths. We show some general ideas, difficulties and suppositions.

\section{Logic of false belief}

\subsection{Alphabet and language}

Our logic is propositional, without quantifiers. We introduce alphabet of our language below. 

\begin{definition}
\wlogic-alphabet consists of:
\begin{enumerate}
\item $PV$ which is a fixed denumerable set of propositional variables $p, q, r, s, ...$.
\item Common logical connectives which are $\land$, $\lor$ and $\rightarrow$.
\item The only derived connective which is $\lnot$ (thus $\lnot \varphi$ is a shortcut for $\varphi \rightarrow \bot$).
\item One modal operator: $\mathsf{W}$. 
\end{enumerate}
\end{definition}

Well-formed formulas are build recursively in a standard style: if $\varphi$, $\psi$ are \textit{wff's} then also $\varphi \lor \psi$, $\varphi \land \psi$, $\varphi \rightarrow \psi$ and $\mathsf{W} \varphi$. Note that $\Leftarrow, \Rightarrow$ and $\Leftrightarrow$ are used only on the level of meta-language (which is classical).

\subsection{Structures and models}

Our initial structure is a pre-ordered neighborhood frame (\pn-frame) defined as it follows:

\begin{definition}
\label{pndef1}
\pn-frame is a triple $F = \langle W, \mathcal{N}, \leq \rangle$ where $\leq$ is a partial order on $W$ and $\mathcal{N}$ is a function from $W$ into $P(P(W))$.
\end{definition}

This definition is very general and it does not provide any relationship between $\leq$ and $\mathcal{N}$. In particular, it will not allow us to speak about monotonicity of forcing with respect to modal formulas. Thus, we shall introduce a particular subclass of \pn-frames.

\begin{definition}
\label{pndef2}
\quad \\
\wpn-frame is a \pn-frame with the following additional restriction:
\begin{equation}
\label{cond1}
[w \leq v, X \in \mathcal{N}_{w}, v \notin X] \Rightarrow X \in \mathcal{N}_{v}.
\end{equation}
\end{definition}

Having structures with appropriate features, we may introduce the notion of model. The first one is general and can be considered as a pattern for the further development of particular models.

\begin{definition}
\label{moddef1}
\pn-model is a quadruple $M = \langle W, \mathcal{N}, \leq, V \rangle$ where $\langle W, \mathcal{N}, \leq \rangle$ is \pn-frame and $V$ is a function from $PV$ into $P(W)$ such that: if $w \in V(q)$ and $w \leq v$ then $v \in V(q)$.
\end{definition}

\begin{definition}
\label{forcedef1}
For every \pn-model $M = \langle W, \mathcal{N}, \leq, V \rangle$, forcing of formulas in a world $w \in W$ is defined inductively:

\begin{enumerate}
\item $w \nVdash \bot$.

\item $w \Vdash q$ $\Leftrightarrow$ $w \in V(q)$ for any $q \in PV$.

\item $w \Vdash \varphi \land \psi$ (resp. $\varphi \lor \psi$) $\Leftrightarrow$ $w \Vdash \varphi$ and (resp. or) $w \Vdash \psi$.

\item $w \Vdash \varphi \rightarrow \psi$ $\Leftrightarrow$ $v \nVdash \varphi$ or $v \Vdash \psi$ for each $v \in W$ such that $w \leq v$.
\end{enumerate}
\end{definition}

We do not make difference between primal valuation and its extended version, using only one symbol $V$. Let us use shortcut $V(\varphi)$ for $\{z \in W; z \Vdash \varphi\}$.

\begin{remark}
Of course, the definition above allows us to say that $w \Vdash \lnot \varphi$ $\Leftrightarrow$ for any $v \geq w$, $v \nVdash \varphi$. 
\end{remark}

Again, we narrow down our initial definition:

\begin{definition}
\label{pndef3}
\wpn-model is a \pn-model with valuation and forcing of non-modal formulas defined just like in Def. \ref{forcedef1} but with an additional clause:

$w \Vdash \mathsf{W} \varphi$ $\Leftrightarrow$ $w \Vdash \lnot \varphi$ and $V(\varphi) \in \mathcal{N}_{w}$.
\end{definition}

\subsection{Monotonicity of forcing}

Here we prove the following fact:

\begin{theorem}
\label{theo1}
In every \wpn-model $M = \langle W, \mathcal{N}, \leq, V \rangle$ the following holds: if $w \Vdash \gamma$ and $w \leq v$, then $v \Vdash \gamma$.
\end{theorem}

\begin{proof}
We shall discuss only the modal case. Assume that $\gamma = \mathsf{W} \varphi$, $w, v \in W$, $w \leq v$ and $w \Vdash \gamma$. Hence, $w \Vdash \lnot \varphi$ and $V(\varphi) \in \mathcal{N}_{w}$. Of course $v \Vdash \lnot \varphi$. In particular, it means that $v \notin V(\varphi) \in \mathcal{N}_{w}$. Now Cond. $\ref{cond1}$ allows us to say that $V(\varphi) \in \mathcal{N}_{v}$. Hence, $w \Vdash \mathsf{W} \varphi$.  
\end{proof}

There is a difference between our definition of $\mathsf{W}$ and the one presented in \cite{fan} or \cite{gilbert}. Their operator was defined as follows: $w \Vdash \mathsf{W}\varphi$ $\Leftrightarrow$ $w \nVdash \varphi$ and $V(\varphi) \in \mathcal{N}_{w}$. However, their framework was classical, so there was no difference between lack of acceptance and acceptance of negation. In our intuitionistic setting, this approach would be problematic: if $w \nVdash \varphi$, then it does not mean that $\varphi$ is denied in each (or even in one) world $v$ placed above $w$. In fact, there is no reason for it: this is the whole difference between classical and intuitionistic negation, at least from the semantical point of view. Hence, we had to modify interpretation of $\mathsf{W}$.

\subsection{Axiomatization}

In this subsection we present sound and complete axiomatization of our basic system.

\begin{definition}
\label{axio1}
\wlogic is defined as the smallest set of formulas containing $\ipc \cup \{\ntax\}$ and closed under the following set of inference rules: $\{\modpon, \rew\}$, where:

\begin{enumerate}
\item \ipc is the set of all intuitionistic axiom schemes and their modal instances (i.e. $\mathsf{W}$-instances).

\item \ntax is the axiom scheme $\mathsf{W} \varphi \rightarrow \lnot \varphi$.

\item \rew is the \textit{rule of extensionality}: $\varphi \leftrightarrow \psi \vdash \mathsf{W} \varphi \leftrightarrow \mathsf{W} \psi$. 

\item \modpon is the rule \emph{modus ponens}: $\varphi, \varphi \to \psi \vdash \psi$. 
\end{enumerate}
\end{definition}

The notion of syntactic consequence (i.e. $\vdash$) is rather standard: if $\Gamma$ is a set of \wlogic-formulas, then $w \vdash \varphi$ \emph{iff} $\varphi$ can be obtained from the finite subset of $\Gamma$ by using axioms of \wlogic and both inference rules. Clearly, if $\varphi \in \Gamma$, then $\Gamma \vdash \varphi$. The same concept of $\vdash$ will be accepted in further systems.

The following theorem holds (and is simple to prove):

\begin{theorem}
\label{sound1}
\wlogic is sound with respect to the class of all \wpn-frames.
\end{theorem}

We can briefly prove completeness of our system with respect to the appropriate class of frames. For brevity, we assume that the reader is aware of the fact that each consistent \wlogic-theory can be extended to the prime theory (this is just an intuitionistic version of the well-known Lindenbaum lemma). Hence, we may go directly to the canonical model. We use the following shortcut: $\widehat{\varphi} = \{z \in W; \varphi \in z\}$. 

For brevity, we start from the general pattern that will be used many times later. 

\begin{definition}
\label{can0}
\canpn-model is a triple $\langle W, \leq, \mathcal{N}, \leq, V \rangle$ where $\mathbf{L}$ may be any logic expressed in \wlogic-language, and:

\begin{enumerate}
\item $W$ is the set of all prime theories of the logic $\mathbf{L}$.

\item For every $w, v \in W$ we say that $w \leq v$ \textit{iff} $w \subseteq v$.

\item $\mathcal{N}$ is a function from $W$ into $P(P(W))$.

\item $V: PV \rightarrow P(W)$ is a function defined as it follows: $w \in V(q) \Leftrightarrow q \in w$.

\end{enumerate}
\end{definition}

Now we may deal with the first particular case:

\begin{definition}
\label{can1}
\wcpn-model is an \canpn-model where $\mathbf{L} = \wlogic$ and for every $w \in W$ and for each formula $\varphi$:

$\mathcal{N}_{w} = \{\widehat{varphi}; \mathsf{W}\varphi \in w\}$. 

\end{definition}

We need the following lemma:

\begin{lemma}
\label{canmonot1}
\wcpn-model is indeed an \wpn-model. 
\end{lemma}

\begin{proof}
In fact, we must check that the monotonicity holds. Let us assume that $w \subseteq v$ and $v \notin X \in \mathcal{N}_{w}$. Now $X = \widehat{\varphi}$ for certain $\varphi$ such that $\mathsf{W} \varphi \in w$. However, $w$ is contained in $v$, hence $\mathsf{W} \varphi \in v$. Thus $\widehat{\varphi} \in \mathcal{N}_{v}$. 
\end{proof}

\begin{remark}
\label{rem1}
Note that we did not use the fact that $v \notin X$. Actually, it was not important. For this reason, we may say that \wcpn-model satisfies even stronger restriction, namely:

\begin{equation}
\label{cond2}
[w \leq v, X \in \mathcal{N}_{w}] \Rightarrow X \in \mathcal{N}_{v}.
\end{equation}
Clearly, our completeness theorem will be true for this smaller class of frames (models). However, we may be interested not only in narrowing down these classes for which completeness result can be proved, but also in broadening those which are \emph{sufficient} for the monotonicity of forcing. 
\end{remark}

Our neighborhood function is well-defined. First, if we assume that $W$ is a collection of all prime theories and $\widehat{\varphi} = \widehat{\psi}$, then we can easily prove that $\vdash \varphi \leftrightarrow \psi$ (using only non-modal tools). Second, assume that $\widehat{\varphi} \in \mathcal{N}_{w}$ and $\widehat{\varphi} = \widehat{\psi}$. If $\widehat{\varphi} \in \mathcal{N}_{w}$, then $\mathsf{W}\varphi \in w$. But $\varphi \leftrightarrow \psi \in \wlogic$ (as we know from the first point of these considerations). Now, by means of $\rew$, $\mathsf{W}\varphi \leftrightarrow \mathsf{W}\psi \in \wlogic \subseteq w$. By \modpon, $\mathsf{W}\psi \in w$. 

Our expected theorem about properties of the canonical model is below:

\begin{theorem}
\label{th1}
Let $M = \langle W, \leq, \mathcal{N}, V \rangle$ be a \wcpn-model. Then for each $\gamma$ and for each $w \in W$ the following holds: $w \Vdash \gamma \Leftrightarrow \gamma \in w$.
\end{theorem}

\begin{proof}
Boolean cases are simple (of course we should remember that implication is intuitionistic). As for the modal case, let us assume that $\gamma = \mathsf{W}\varphi$. 

$(\Rightarrow)$

Assume that $w \Vdash \gamma$. Hence, $w \Vdash \lnot \varphi$ and $V(\varphi) \in \mathcal{N}_{w}$. By induction hypothesis $\widehat{\varphi} \in \mathcal{N}_{w}$. But then, by the very definition of canonical neighborhood, $\mathsf{W}\varphi \in w$. 

$(\Leftarrow)$

Let $\mathsf{W}\varphi \in w$. By means of \ntax and \modpon we infer that $\lnot \varphi \in w$. This is Boolean case: we are able to prove in a standard manner\footnote{We may use the fact that $\lnot \varphi$ can be written as $\varphi \to \bot$.} that $w \Vdash \lnot \varphi$. From the definition of canonical neighborhood we have that $\widehat{\varphi} \in W$ which means (here we use induction hypothesis) that $V(\varphi) \in \mathcal{N}_{w}$. We sum up our results to say that $w \Vdash \mathsf{W}\varphi$. 
\end{proof}

Now it is easy to formulate theorem about completeness:

\begin{theorem}
\wlogic is strongly complete with respect to the class of all \wpn-frames; and also with respect to those in which neighborhood function satisfies Cond. \ref{cond2}.
\end{theorem}

The proof of the theorem above is standard and based on the assumption that $w$ is a theory and $w \nvdash \varphi$. The idea is to show that there is a prime theory $v$ in a canonical model such that $w \subseteq v$ and $w \nVdash \varphi$. 

One could say that our logic is not a proper system of "false belief" because it is (in its modal aspect) much weaker than some systems studied in \cite{steinsvold}, \cite{gilbert} and \cite{fan}. This will be discussed in the next subsection.

\subsection{Stronger systems of false belief}

It is not difficult to add one very natural axiom to our initial kit, namely \wcax: $(\mathsf{W}\varphi \land \mathsf{W}\psi) \to \mathsf{W}(\varphi \land \psi)$. Here are necessary definitions:

\begin{definition}
\label{axio2}
\wwlogic is defined as $\wlogic \cup \{\wcax\}$. 
\end{definition}

\begin{definition}
\label{pndef4}
\wwpn-model is defined as \wpn-model with one additional clause (the one of \emph{closure under binary intersections}):
\begin{equation}
\label{cond4}
[X, Y \in \mathcal{N}_{w}] \Rightarrow X \cap Y \in \mathcal{N}_{w}. 
\end{equation}
\end{definition}

Canonical model for \wwlogic (i.e. \wwcpn-model) is defined exactly in the same way as \wcpn-model (but its worlds are prime theories of \wwlogic). Thus, we should only prove the following lemma:

\begin{lemma}
\label{canmonot2}
\wwcpn-model is indeed an \wwpn-model.
\end{lemma}

\begin{proof}
This is simple. Assume that $X, Y \in \mathcal{N}_{w}$. Hence, $X = \widehat{\varphi}$ and $Y = \widehat{\psi}$ (for certain $\varphi$ and $\psi$ such that $\mathsf{W}\varphi \in w$ and $\mathsf{W}\psi \in w$). Then $X \cap Y = \widehat{\varphi} \cap \widehat{\psi} = \widehat{\varphi \land \psi}$. At the same time, we use axiom \wcax to say that $\mathsf{W}(\varphi \land \psi) \in w$. Thus $X \cap Y \in \mathcal{N}_{w}$. 
\end{proof}

Now we can say that:
\begin{theorem}
\wwlogic is strongly complete with respect to the class of all \wwpn-frames (and those \wwpn-frames which satisfy Cond. \ref{cond2}).
\end{theorem}

Let us introduce another system: it will be an intuitionistic version of $\mathbf{M}^{\mathsf{W}}$ studied in \cite{fan}.\footnote{Precisely speaking, Fan used axiom $\mathsf{W}(\varphi \land \psi) \land \lnot \psi \to \mathsf{W}\psi$. The rule \rew can be derived from this axiom and \ntax.}

\begin{definition}
\label{axio2}
\wwwlogic is defined as $\wlogic \cup \{\wnec\}$, where:

\begin{enumerate}
\item \wnec is the rule $\varphi \to \psi \vdash (\mathsf{W}\varphi \land \lnot \psi) \to \mathsf{W}\psi)$. 
\end{enumerate}
\end{definition}

We introduce new kind of models:

\begin{definition}
\wwwpn-model is an \wpn-model with one additional clause (the one of \emph{supplementation}):
\begin{equation}
\label{cond3}
[X \in \mathcal{N}_{w}, X \subseteq Y] \Rightarrow Y \in \mathcal{N}_{w}.
\end{equation}
\end{definition}

\begin{lemma}
\wwwlogic is sound with respect to the class of all \wwwpn-models. 
\end{lemma}

\begin{proof}
We shall check only \wnec. Assume that $\varphi \rightarrow \psi$ is globally true. Suppose that there are $M = \langle W, \leq, \mathcal{N}, V\rangle$ and $w \in W$ such that $w \nVdash (\mathsf{W}\varphi \land \lnot \psi) \to \mathsf{W}\psi)$. Hence, there is $v \geq w$ such that $v \Vdash (\mathsf{W}\varphi \land \lnot \psi)$ but $v \nVdash \mathsf{W}\psi$. This means that: \textbf{i)} $v \Vdash \lnot \varphi$, $V(\varphi) \in \mathcal{N}_{v}$, $v \Vdash \lnot \psi$; and \textbf{ii)} $v \nVdash \lnot \psi$ or $V(\psi) \notin \mathcal{N}_{v}$. It is not possible that $v \nVdash \lnot \psi$. On the other hand, if $\varphi \rightarrow \psi$ is globally true, then $V(\varphi) \subseteq V(\psi)$. Supplementation allows us to say that $V(\psi) \in \mathcal{N}_{v}$. This is contradiction.
\end{proof}

Let us go to the canonical model.

\begin{definition}
\label{can2}
\wwwcpn-model is an \canpn-model where $\mathbf{L} = \wwwlogic$ and for every $w \in W$ and for each formula $\varphi$:

$\mathcal{N}_{w} = \{X \subseteq W; \text{ there is } Y \in n_{w} \text{ such that } Y \subseteq X\}, \text{ where } n_{w} = \{\widehat{\varphi}; \mathsf{W}\varphi \in w\}$

\end{definition}

We must prove the following lemma:

\begin{lemma}
\label{canmonot3}
\wwwcpn-model is indeed an \wwwpn-model. 
\end{lemma}

\begin{proof}
Let us think about monotonicity condition, namely Cond. \ref{cond1}. Assume that $w \subseteq v$, $v \notin X$ and $X \in \mathcal{N}_{w}$. Now there is $Y \in n_{w}$ such that $Y \subseteq X$. However, as we already know, function $n$ satisfies Cond. \ref{cond2} which is even stronger than Cond. \ref{cond1}. Hence $Y \in n_{v}$ and thus $X \in \mathcal{N}_{v}$. 

As for the supplementation, it is obvious by the very definition of \wwwcpn-model. Assume that $X \in \mathcal{N}_{w}$ and $X \subseteq Y$. Then there is $S \in n_{w}$ such that $S \subseteq X \subseteq Y$. We are ready.
\end{proof}

\begin{remark}
Note that in fact \wwwcpn-model satisfies stronger monotonicity condition, i.e. Cond. \ref{cond2}. As in the case of \wcpn-model, we did not use the fact that $v \notin X$. 
\end{remark}

The next theorem is crucial for completeness:

\begin{theorem}
Let $M = \langle W, \leq, \mathcal{N}, V \rangle$ be an \wwwcpn-model. Then for each $\gamma$ and for each $w \in W$ the following holds: $w \Vdash \gamma \Leftrightarrow \gamma \in w$.
\end{theorem}

\begin{proof}
We consider the modal case. Assume that $\gamma = \mathsf{W}\varphi$.

$(\Rightarrow)$

Let $w \Vdash \mathsf{W}\varphi$. Then $w \Vdash \lnot \varphi$ and $V(\varphi) \in \mathcal{N}_{w}$. Hence $\lnot \varphi \in w$ and $\widehat{\varphi} \in \mathcal{N}_{w}$. Also there is $\widehat{\psi} \in n_{w}$ such that $\widehat{\psi} \subseteq \widehat{\varphi}$ and $\mathsf{W}\psi \in w$. Now $\vdash \psi \to \varphi$, i.e. this formula is a theorem. Hence, (by means of \wnec) $\vdash (\mathsf{W}\psi \land \lnot \varphi) \to \mathsf{W}\varphi$. Assume now that $\mathsf{W}\varphi \notin w$. There are two possible reasons. First: $\mathsf{W}\psi \notin w$ (contradiction). Second: $\lnot \varphi \notin w$. But $\lnot \varphi \in w$, as we already know. 

$(\Leftarrow)$

Assume that $\mathsf{W}\varphi \in w$. Now $\lnot \varphi \in w$ and then $w \Vdash \lnot \varphi$. Then $\widehat{\varphi} \in \mathcal{N}_{w}$. By induction hypothesis, $V(\varphi) \in \mathcal{N}_{w}$. Thus, $w \Vdash \mathsf{W} \varphi$. 
\end{proof}

\begin{theorem}
\wwwlogic is strongly complete with respect to the class of all \wwwpn-models (and those \wwwpn-models which satisfy Cond. \ref{cond2}).
\end{theorem}

Let us sum up these results.

\begin{definition}
\label{axio3}
\wwwwlogic is defined as $\wwwlogic \cup \{\wcax\}$.
\end{definition}

\begin{definition}
\label{pndef5}
\wwwwpn-model is an \wpn-model satisfying both supplementation and closure under binary intersections.
\end{definition}

\begin{definition}
\wwwwcpn-model is defined just like \wwwcpn-model (but $W$ consists of \wwwwlogic prime theories).
\end{definition}

\begin{theorem}
\wwwwlogic is strongly complete with respect to the class of all supplemented \wwwwpn-models closed under binary intersections (and those of them which satisfy Cond. \ref{cond2}). 
\end{theorem}

\begin{proof}
It is enough to check closure under intersections in \wwwwcpn-model. Let $X, Y \in \mathcal{N}_{w}$. Hence, there are $\widehat{\varphi}, \widehat{\psi} \in n_{w}$ such that $\widehat{\varphi} \subseteq X$ and $\widehat{\psi} \subseteq Y$. Clearly, $\widehat{\varphi} \cap \widehat{\psi} = \widehat{\varphi \land \psi} \subseteq X \cap Y$. Also, by means of \wcax, $\mathsf{W}(\varphi \land \psi) \in w$ and thus $X \cap Y \in n_{w}$. Finally, $X \cap Y \in \mathcal{N}_{w}$. 
\end{proof}

There are some issues which should be discussed. This will be done in the next subsection.

\subsection{About canonical models and monotonicity}

In general, we borrowed some ideas from \cite{fan} and \cite{gilbert}. However, there are some subtle differences. Let us resume the line of thought presented in \cite{fan} with respect to the classical version of \wwwlogic, that is $\mathbf{M}^{\mathsf{W}}$. 

$\mathbf{i)}$ Fan assumed that the canonical model for $\mathbf{M}^{\mathsf{W}}$ is any model based on maximal theories \footnote{With valuation defined as usual.} in which $\mathsf{W}\varphi \lor \varphi \in w$ $\Leftrightarrow$ $\widehat{\varphi} \in \mathcal{N}_{w}$ [*]. Let us define also another condition for the further needs: $\mathsf{W}\varphi \in w$ $\Leftrightarrow$ $\widehat{\varphi} \in \mathcal{N}_{w}$ [**].

Thus, he has defined the whole family of such models (from the minimal to the maximal one; the former contains precisely proof-sets\footnote{Such proof-sets $\widehat{\varphi}$ that $\mathsf{W}\varphi \lor \varphi \in w$, of course.}, the latter consists of proof-sets \emph{and} all non-proof-sets). 

Whereas we defined our $n_{w}$ (for any $w \in W$) precisely just like in the minimal model. We said that $n_{w}$ contains \emph{only} those $\widehat{\varphi}$ for which $\mathsf{W}\varphi \in w$. Also, our line of reasoning was closer to [**] than to [*] but this is not crucial here.

$\mathbf{ii)}$ Then Fan introduced the notion of supplemented canonical model $M^{+}$ (supplementation of canonical model $M$, in other words) in which $\mathcal{N}_{w}^{+} = \{X \subseteq W$; there is $Y \in \mathcal{N}_{w}$ such that $Y \subseteq X\}$. He showed that $M^{+}$ is indeed canonical: that it satisfies [*]. Due to some reasons, it would be problematic for him to show that $M^{+}$ satisfies [**]. It would require typical monotonicity rule $\varphi \to \psi \vdash \mathsf{W}\varphi \to \mathsf{W}\psi$ which is not sound. 

Our way is different. We do not say that neighborhood function $\mathcal{N}$ in our \wwwcpn-model is "canonical" in the same sense as $n$. It would be irrelevant because the definition of $n_{w}$ leaves no place for any variants: as we said, these are \emph{precisely} proof-sets satisfying certain property. However, maybe it would be sensible to follow Fan directly? Assume that $n_{w}$ is defined by means of, let us say, clause [**]. It can be [*] also, it does not matter: the real problem lies in the persistence of truth. Both [*] and [**] are too vague to force monotonicity (with respect to $n$ and in the sense of Cond. \ref{cond1} or Cond. \ref{cond2}). Assume that $w \subseteq v$ and $X \in n_{w}$. We may also suppose that $v \notin X$. If $X = \widehat{\varphi}$ for certain $\varphi$, then we can repeat our actual reasoning. But if $X \neq \widehat{\varphi}$ for any $\varphi$, then we cannot say anything special about this fact. Of course, if our model was maximal, then by the very definition $X$ would belong to $n_{w}$. But if not, then we would be in a quandary.

It seems that a similar solution to a similar dilemma has been obtained in \cite{dalmonte}. These authors prepared bi-neighborhood semantics for weak intuitionistic modal logics and they also used minimal canonical models. In case of richer logics they used canonical models "equipped with" supplementation (just as our \wwwcpn-model), not the supplementation of previously defined model. 

Gilbert and Venturi found different solution than Fan did. They assumed that neighborhoods in canonical model (for the classical version of \wwwwlogic) are defined by means of [**]. Then they used \emph{negative} supplementation. This is the following condition: 

$Y \in \mathcal{N}_{w}$, $Y \subseteq X$, $w \notin X$ $\Rightarrow$ $X \in \mathcal{N}_{w}$.

In the negative supplementation of canonical model, for any $w \in W$ and for each $\varphi$ we have:

$\mathcal{N}_{w}^{+} = \{X \subseteq W;$ there is $Y \in \mathcal{N}_{w}$ such that $Y \subseteq X$ and $w \notin X\}$.

Again, this "feature of negativity" (that is, the assumption that $w \notin X$) is helpful in proving that negative supplementation is indeed canonical. From our point of view, one thing is interesting. Let us reproduce the definition of \wwwwcpn-model but with the following definition of neighborhoods:

$\mathcal{N}_{w} = \{X \subseteq W;$ there is $Y \in n_{w}$ such that $Y \subseteq X$ and $w \notin X$\}, where $n_{w} = \{ \widehat{\varphi}; \mathsf{W}\varphi \in w\}$.

This is in accordance with our previous considerations. The whole proof of completeness is almost identical. However, there is one noteworthy moment. Let us prove that Cond. \ref{cond1} of monotonicity is satisfied. Let $w \subseteq v$, $X \in \mathcal{N}_{w}$ and $v \notin X$. There is $Y \in n_{w}$ such that $Y \subseteq X$. However, $n$ satisfies Cond. \ref{cond2} (Lem. \ref{canmonot1} and Rem. \ref{rem1}), so $Y \in n_{v}$. Thus $X \in \mathcal{N}_{v}$. Note that in this case we prove only that $\mathcal{N}$ satisfies Cond. \ref{cond1} and not necessarily Cond. \ref{cond2}. Clearly, we used the assumption that $v \notin X$. 

\subsection{Confirmation and discouragement}

\subsection{Language, frames and models}
In this section we shall work with all languages (resp. propositional systems and classes of frames or models) simultaneously. Some definitions and proofs will be shortened to avoid repetition of obvious things. 

\begin{definition}
\alog (resp. \blog, \clog, \dlog)-alphabet in its non-modal part is identical with \wlogic-alphabet. Instead of $\mathsf{W}$, it contains the following modal operator:

\begin{multicols}{2}
\begin{itemize}
\item[$\mathbf{i)}$] $\Box$ for \alog.
\item[$\mathbf{ii)}$] $\Diamond$ for \blog.
\item[$\mathbf{iii)}$] $\blacksquare$ for \clog.
\item[$\mathbf{iv)}$] $\bullet$ for \dlog.
\end{itemize}
\end{multicols}
\end{definition}

\begin{definition}
\label{pndefi}
\apn (resp. \bpn, \cpn, \dpn) frame is a \pn-frame with the following additional restriction:

(for \apn-frames)
\begin{equation}
\label{condi1}
[w \leq v, v \in X \subseteq W, X \in \mathcal{N}_{w}] \Rightarrow X \in \mathcal{N}_{v}. 
\end{equation}

(for \bpn-frames)
\begin{equation}
\label{condi2}
[w \leq v, v \in X \subseteq W, -X \notin \mathcal{N}_{w}] \Rightarrow -X \notin \mathcal{N}_{v}. 
\end{equation}

(for \cpn-frames)
\begin{equation}
\label{condi3}
[w \leq v, v \in X \subseteq W, -X \in \mathcal{N}_{w}] \Rightarrow -X \in \mathcal{N}_{v}.
\end{equation}

(for \dpn-frames)
\begin{equation}
\label{condi4}
[w \leq v, v \in X \subseteq W, X \notin \mathcal{N}_{w}] \Rightarrow X \notin \mathcal{N}_{v}.
\end{equation}
\end{definition}

\begin{definition}
\apn (resp. \bpn, \cpn, \dpn)-model is a \pn-model with valuation and forcing of non-modal formulas defined just like in Def. \ref{forcedef1} but with an additional clause:

\begin{itemize}
\item[$\mathbf{i)}$] $w \Vdash \Box \varphi$ $\Leftrightarrow$ $w \Vdash \varphi$ and $V(\varphi) \in \mathcal{N}_{v}$ (in \apn-model).

\item[$\mathbf{ii)}$] $w \Vdash \Diamond \varphi$ $\Leftrightarrow$ $w \Vdash \varphi$ and $-V(\varphi) \notin \mathcal{N}_{w}$ (in \bpn-model).

\item[$\mathbf{iii)}$] $w \Vdash \blacksquare \varphi$ $\Leftrightarrow$ $w \Vdash \varphi$ and $-V(\varphi) \in \mathcal{N}_{w}$ (in \cpn-model).

\item[$\mathbf{iv)}$] $w \Vdash \bullet \varphi$ $\Leftrightarrow$ $w \Vdash \varphi$ and $V(\varphi) \notin \mathcal{N}_{w}$ (in \dpn-model).
\end{itemize}
\end{definition}

Here we may discuss our interpretation of those operators. \emph{First}, we have $\Box$ which is often used to model \emph{necessity}. Indeed, what we are doing here, is just an application of the general idea of necessity in the intuitionistic framework. However, our understanding of $\Box$ can be expressed in this way: our agent accepts $\varphi$ (in a certain set of circumstances, namely in a world $w$) and, additionally, his decision is supported or confirmed by the fact that $V(\varphi) \in \mathcal{N}_{w}$. 

\emph{Second}, we have $\Diamond$. This symbol is often used to speak about \emph{possibility}. However, here we assume that forcing of $\Diamond \varphi$ implies (in particular) forcing of $\varphi$. In the light of our interpretation, we have the following situation: our agent accepts $\varphi$ and we \emph{cannot} say that he is dissuaded from this decision: the set $-V(\varphi)$ is not among $w$-neighborhoods. 

Let us assume for a moment that our language contains both $\Box$ and $\Diamond$ and our model satisfies both needful conditions of monotonicity (i.e. it is $\Box \Diamond$\pn-model). Now we may easily show that $\Box \varphi \rightarrow \lnot \Diamond \lnot \varphi$ is a tautology. Assume that $w$ denies this formula. Then there is $v \geq w$ such that $v \Vdash \Box \varphi$ and $v \nVdash \lnot \Diamond \lnot \varphi$. Hence, $v \Vdash \varphi$ and $V(\varphi) \in \mathcal{N}_{v}$. However, there is $s \geq v$ such that $s \Vdash \Diamond \lnot \varphi$. In particular, it means that $s \Vdash \lnot \varphi$. But this is impossible because $s \geq v$ and we have persistence of truth.

On the other hand, let us consider $W = \{w, v\}$, where $\mathcal{N}_{w} = \mathcal{N}_{v} = W$, $w \leq v$ and $V(\varphi) = \{v\}$. This model satisfies both required monotonicity conditions (namely, for $\Box$ and $\Diamond$): \textbf{i)} if $x \leq y$ and $y \in X \in \mathcal{N}_{x}$, then $X = W$, hence $X \in \mathcal{N}_{y}$; \textbf{ii)} if $x \leq y$, $y \in X$ and $-X \notin \mathcal{N}_{x}$, then $-X \notin \mathcal{N}_{y}$, because if not, then $-X = W$, hence $X = \emptyset$ (contradiction, due to the fact that $x \in X$). Anyway, $w \Vdash \lnot \Diamond \lnot \varphi$, because for any $x \geq w$ we have $x \nVdash \lnot \varphi$. At the same time, $w \nVdash \Box \varphi$ because $w \nVdash \varphi$. Hence, $\lnot \Diamond \lnot \varphi \rightarrow \Box \varphi$ is not a tautology in the class of $\Box \Diamond$\pn-models. 

Note that $\Box \varphi \rightarrow \Diamond \varphi$ is not a tautology in this class, even if this implication is very natural. Clearly, it will become true if we assume that $X \in \mathcal{N}_{w} \Rightarrow -X \notin \mathcal{N}_{w}$ (for any $w \in W$). 

\emph{Third}, we have $\blacksquare$. This operator says: $\varphi$ is accepted by the subject but he is advised against this formula (by his advisors). In the \emph{fourth} case we have $\bullet$: now $\varphi$ is still accepted by our agent but it is not suggested by his "advisory board". In fact, there is also another interpretation of $\bullet$: that $\varphi$ is true but unknown. This leads us to the \emph{logic of unknown truths}. We shall deal with this topic later. 

The question is: how to combine all four operators in one propositional logic based on intuitionism? Of course, we speak about logic complete with respect to a certain class of frames: relational-neighborhood or different. In fact, it seems that such system would split into two parts: $\Box \Diamond$-one (with axiom $\Box \varphi \to \Diamond \varphi$) and $\blacksquare \bullet$-one (with axiom $\blacksquare \varphi \to \bullet \varphi$). From the purely philosophical point of view it is not clear if there should be any close connection between these two parts. 

In general, it seems that relationships between our operators should be (or, at least, can be) rather weak. This is because we should allow our "advisory board" to be neutral in some sense. For example, if we are advised against $\varphi$ (as in  the case of $\blacksquare \varphi$), then it does not mean that we are encouraged to accept $\lnot \varphi$. Not on the ground of intuitionistic logic or some other logics with non-classical negation. Moreover, if we are \emph{not} discouraged ($\Diamond \varphi$ applies to us), then it still does not mean that we are encouraged (by means of $\Box \varphi$). If $\varphi$ is not suggested by our board ($\bullet \varphi$ case), then it does not mean that this board is against $\varphi$ (as in the case of $\blacksquare \varphi$). 

Note that we can speak here about three different things: \textbf{i)} lack of recommendation of some formula; \textbf{ii)} advising against this formula; \textbf{iii)} supporting the opposite formula, namely $\lnot \varphi$. This last case was not analysed above. It would require a new operator, say $\mathbf{N}$, defined in the following manner: $w \Vdash \mathbf{N}$ $\Leftrightarrow$ $w \Vdash \varphi$ and $V(\lnot \varphi) \in \mathcal{N}_{w}$. These considerations are interesting on the intuitionistic ground: surely, $V(\lnot \varphi)$ may be different than $-V(\varphi)$. However, it is always true that $V(\lnot \varphi) \subseteq -V(\varphi)$. As for the monotonicity condition $\mathbf{N}$, it can be: 

[$w \leq v$, $v \in X$ and there is $Y \in \mathcal{N}_{w}$ such that $Y \subseteq -X]$ $\Rightarrow$ $Y \in \mathcal{N}_{v}$.

Suppose now that $w \in W$ and $w \Vdash \mathbf{N}\varphi$. Thus $w \Vdash \varphi$ and $V(\lnot \varphi) \in \mathcal{N}_{w}$. Of course $v \Vdash \varphi$ and $V(\lnot \varphi) \subseteq -V(\varphi)$.  Hence, $V(\lnot \varphi) \in \mathcal{N}_{v}$. Now $v \Vdash \mathbf{N}\varphi$.

Clearly, it is possible to find three new versions of $\mathbf{N}$, defined analogously to $\Diamond$, $\blacksquare$ and $\bullet$. 

\subsection{Monotonicity and completeness}

One can easily prove the following fact:

\begin{theorem}
In every \apn (resp. \bpn, \cpn, \dpn)-model $M = \langle W, \mathcal{N}, \leq, V \rangle$ the following holds: if $w \Vdash \gamma$ and $w \leq v$, then $v \Vdash \gamma$.
\end{theorem}

\begin{proof}
It is easy to see that the conditions imposed on models are coherent with monotonicity. We shall not discuss all cases; let us show the idea on the example of $\bullet$. Assume that we have \dpn-model, $w \in W$, $w \leq v$. Let $\gamma = \Box \varphi$ and $w \Vdash \gamma$. Then $w \Vdash \varphi$ and $V(\varphi) \notin \mathcal{N}_{w}$. Then $v \Vdash \varphi$, i.e. $v \in V(\varphi) \notin \mathcal{N}_{w}$. From the monotonicity condition, $V(\varphi) \notin \mathcal{N}_{v}$. Thus $v \Vdash \Box \varphi$. 
\end{proof}

Now we present sound and complete axiomatization of our logics. In fact, they can all be considered as one system: slightly stronger than the weakest intuitionistic mono-modal logic. Hence, one can say that what we are really doing (at least on this stage of our investigations) is to present four different, sound and complete, types of semantics for one propositional system. However, due to the fact that we want to distinguish among these four interpretations which were studied before, we use four symbols. 

\begin{definition}
\label{axiom1}
The \alog (resp. \blog, \clog, \dlog)-logic is defined as the smallest set of formulas containing $\ipc$ and closed under the following set of rules: $\{\tax x, \modpon, \ext x\}$ where:

\begin{enumerate}

\item \ipc is the set of all intuitionistic axiom schemes and their modal instances.

\item $\tax_x$ is the axiom scheme $x \varphi \rightarrow \varphi$, where:

\begin{multicols}{2}
\begin{itemize}
\item [$\mathbf{i)}$] $x = \Box$, if logic is \alog.

\item [$\mathbf{ii)}$] $x = \Diamond$, if logic is \blog.

\item [$\mathbf{iii)}$] $x = \blacksquare$, if logic is \clog.

\item [$\mathbf{iv)}$] $x = \bullet$, if logic is \dlog.

\end{itemize}
\end{multicols}

\item \modpon is \textit{modus ponens}.

\item $\ext_x$ is \textit{rule of extensionality}: $\varphi \leftrightarrow \psi \vdash x \varphi \leftrightarrow x \psi$, where $x$ defined just as in the case of $\tax_x$.
\end{enumerate}
\end{definition}

The following theorem holds (and is obvious):

\begin{theorem}
\label{sound}
\alog (resp. \blog, \clog, \dlog) is sound with respect to the class of all \apn (resp. \bpn, \cpn, \dpn)-frames.
\end{theorem}

As for the canonical model, we define it simultaneously for each system.

\begin{definition}
\label{canmod1}
\acanpn (resp. \bcanpn, \ccanpn, \dcanpn)-model is a triple $\langle W, \leq, \mathcal{N}, \leq, V \rangle$ where:

\begin{enumerate}
\item $W$ is the set of all \alog (resp. \blog, \clog, \dlog) prime theories.

\item For every $w, v \in W$ we say that $w \leq v$ \textit{iff} $w \subseteq v$.

\item $\mathcal{N}$ is a function from $W$ into $P(P(W))$ such that for every $w \in W$ and for each formula $\varphi$:

\begin{itemize}
\item[$\mathbf{i)}$] $\mathcal{N}_{w} = \{ \widehat{\varphi}; \Box \varphi \in w \}$ (in \acanpn-model).

\item[$\mathbf{ii)}$] $\mathcal{N}_{w} = \{X \subseteq W; X = W \setminus \widehat{\varphi}; \Diamond \varphi \notin w \}$ (in \bcanpn-model).

\item[$\mathbf{iii)}$] $\mathcal{N}_{w} = \{ X \subseteq W; X = W \setminus \widehat{\varphi}; \blacksquare \varphi \in w \}$ (in \ccanpn-model).

\item[$\mathbf{iv)}$] $\mathcal{N}_{w} = \{\widehat{\varphi}; \bullet \varphi \notin w \}$ (in \dcanpn-model).
\end{itemize}

\item $V: PV \rightarrow P(W)$ is a function defined as it follows: $w \in V(q) \Leftrightarrow q \in w$.

\end{enumerate}
\end{definition}

For convenience, let us agree on the following denotations: 

$\mathbf{CAN} = \{\acanpn, \bcanpn, \ccanpn, \dcanpn\}$,

$\mathbf{LOG} = \{\alog, \blog, \clog, \dlog\}$.

We must check if our (canonical) neighborhood functions are \textit{well-defined}.

\begin{lemma}
\label{lemcan}
For each $M \in \mathbf{CAN}$ and for each prime theory $w$ based on the appropriate logic from $\mathbf{LOG}$ we have the following:

\begin{itemize}

\item In $\acanpn$-model:  if $\widehat{\varphi} \in \mathcal{N}_{w}$ and $\widehat{\varphi} = \widehat{\psi}$, then $\Box \psi \in w$.

\item In $\bcanpn$-model:  if $W \setminus \widehat{\varphi} \in \mathcal{N}_{w}$ and $\widehat{\varphi} = \widehat{\psi}$, then $\Diamond \psi \notin w$.

\item In $\ccanpn$-model:  if $\widehat{\varphi} \in \mathcal{N}_{w}$ and $\widehat{\varphi} = \widehat{\psi}$, then $\blacksquare \psi \notin w$.

\item In $\dcanpn$-model:  if $W \setminus \widehat{\varphi} \in \mathcal{N}_{w}$ and $\widehat{\varphi} =\widehat{\psi}$, then $\bullet \psi \notin w$.

\end{itemize}
\end{lemma}

\begin{proof}
We shall not deal with all cases (they are similar). Take, for example, \ccanpn-model. If $\widehat{\varphi} \in \mathcal{N}_{w}$, then $\blacksquare \varphi \notin w$. As we already know, $\varphi \leftrightarrow \psi \in \clog$. Hence, $\blacksquare \varphi \leftrightarrow \blacksquare \psi \in \clog \subseteq w$. Suppose that $\blacksquare \psi \in w$. Then, by means of \modpon and the fact that $\blacksquare \psi \rightarrow \blacksquare \varphi$ is a theorem, we have $\blacksquare \varphi \in w$. Contradiction. Hence, $\blacksquare \psi \notin w$. 
\end{proof}

 The next lemma deals with monotonicity.

\begin{lemma}
\label{monotone}
Canonical models from $\mathbf{CAN}$ satisfy conditions (resp.) $\mathbf{i)}$ - $\mathbf{iv)}$ from Def. \ref{pndefi}. 
\end{lemma}

\begin{proof}
Again, we shall expose only one case here. Let $M$ be a \dcanpn-model. Assume that $w \subseteq v \in X$ and $X \notin \mathcal{N}_{w}$. Hence, for any $\varphi$ we have: $X \neq \widehat{\varphi}$ or $\bullet \varphi \in w$. In the first case, we are done, so we focus on the second case. As we know, $w \subseteq v$, hence $\bullet \varphi \in v$. Thus, $X$ cannot be in $\mathcal{N}_{v}$, even if it has the form of $\widehat{\varphi}$. 
\end{proof}

Finally, we obtain the crucial lemma:

\begin{lemma}(\textit{truth lemma})
\label{truth}
In any model from $\mathbf{CAN}$ we have (for each $\gamma$ and for each $w \in W$): $w \Vdash \gamma \Leftrightarrow \gamma \in w$.
\end{lemma}

\begin{proof}
Let us think about \dcanpn-model, for example. Assume that $\gamma = \bullet \varphi$ and $w \in W$. Let $w \Vdash \gamma$, i.e. $w \Vdash \varphi$ and $V(\varphi) \notin \mathcal{N}_{w}$. By induction hypothesis, $\varphi \in w$ and $\widehat{\varphi} \notin \mathcal{N}_{w}$. By the very definition of canonical neighborhoods, $\bullet \varphi \in w$. 

Now suppose that $\bullet \varphi \in w$. Then $\varphi \in w$ (hence, by induction, $w \Vdash \varphi$) and $\widehat{\varphi} \notin \mathcal{N}_{w}$. By induction, $V(\varphi) \notin \mathcal{N}_{w}$. We are ready.
\end{proof}

Now we obtain our expected conclusion: 

\begin{theorem}
Each system from $\mathbf{LOG}$ is strongly complete with respect to the appropriate class of models from $\mathbf{CAN}$. 
\end{theorem}

\begin{remark}
Note that in Lemma \ref{monotone} we did not use the fact that $v \in X$. Hence, in each case canonical model satisfies more restricted condition of monotonicity. For example, in case of \acanpn-model it would be: $[w \leq v$, $X \in \mathcal{N}_{w}]$ $\Rightarrow$ $X \in \mathcal{N}_{v}$. Analogous clauses can be easily derived for other models. For this reasons, our systems (which are, as we said, one system from the purely syntactical point of view) are complete with respect to the narrower classes.
\end{remark}

As we have already said, $\bullet$ has certain connection with so called logics of unknown truths, investigated in \cite{gilbert} and \cite{fan}. In fact, this operator is taken from these systems. Of course, they are basically classical. It is possible to use another operator, namely $\circ \varphi$, defined as $\lnot \bullet \varphi$, i.e. $w \Vdash \circ \varphi$ $\Leftrightarrow$ $w \nVdash \varphi$ or $V(\varphi) \in \mathcal{N}_{w}$. This duality between $\circ$ and $\bullet$ is true also in the intuitionistic setting, as one can easily check using \dpn-models. 

Fan studied some properties of languages containing both $\bullet$ and $\varphi$. Hence, we think that these issues should be studied also in the context of intuitionism. Our initial research suggests that it may be interesting. For example, consider system $\mathbf{B}_{k}$, investigated in \cite{gilbert}. It contains axioms $\circ \top$, $\bullet \varphi \to \varphi$ and $(\circ \varphi \land \circ \psi) \to \circ(\varphi \land \psi)$. It is closed under \modpon and the rule $\varphi \to \psi \vdash (\circ \varphi \land \varphi) \to (\circ \psi \land \psi)$. It seems that in the proof of completeness of this system (with respect to the class of filters) Gilbert and Venturi used the fact that if $\varphi \notin w$ (where $w$ is a maximal theory), then $\lnot \varphi \in w$ and hence $\circ \varphi \in w$ (because from the axiom $\bullet \varphi \to \varphi$ we may infer that $\lnot \varphi \to \lnot \bullet \varphi \in w$, i.e. $\lnot \varphi \to \circ \varphi \in w$). However, in intuitionistic prime theories we cannot say that if $\varphi \notin w$, then $\lnot \varphi \in w$.

\end{document}